\documentclass{amsart}
\usepackage{amsmath}
\usepackage{amsfonts}
\usepackage{enumerate}
\usepackage{amssymb}
\usepackage{amscd}
\usepackage[all]{xy}
\usepackage{bm}

\DeclareMathOperator{\gr}{gr} 
\DeclareMathOperator{\Der}{Der}

\DeclareMathOperator{\SL}{SL}
\DeclareMathOperator{\GL}{GL}

\DeclareMathOperator{\Ann}{Ann}
\DeclareMathOperator{\End}{End}
\DeclareMathOperator{\Sym}{Sym}
\DeclareMathOperator{\Ch}{Ch}
\DeclareMathOperator{\Aut}{Aut}

\DeclareMathOperator{\Spec}{Spec}
\DeclareMathOperator{\hgt}{ht}

\begin{document}
\theoremstyle{plain}
\newtheorem{MainThm}{Theorem}
\renewcommand{\theMainThm}{\Alph{MainThm}}
\newtheorem{MainCor}{Corollary}
\renewcommand{\theMainCor}{\Alph{MainCor}}
\newtheorem*{trm}{Theorem}
\newtheorem*{lem}{Lemma}
\newtheorem*{prop}{Proposition}
\newtheorem*{defn}{Definition}
\newtheorem*{thm}{Theorem}
\newtheorem*{example}{Example}
\newtheorem*{cor}{Corollary}
\newtheorem*{conj}{Conjecture}
\newtheorem*{hyp}{Hypothesis}
\newtheorem*{thrm}{Theorem}
\newtheorem*{quest}{Question}
\theoremstyle{remark}
\newtheorem*{rem}{Remark}
\newtheorem*{rems}{Remarks}
\newtheorem*{notn}{Notation}
\newcommand{\Fp}{\mathbb{F}_p}
\newcommand{\Fq}{\mathbb{F}_q}
\newcommand{\Zp}{\mathbb{Z}_p}
\newcommand{\Qp}{\mathbb{Q}_p}
\newcommand{\invlim}{\lim\limits_{\longleftarrow}}
\newcommand{\Frac}{\operatorname{Frac}}
\newcommand{\Kdim}{\operatorname{Kdim}}
\newcommand{\GKdim}{\operatorname{GKdim}}
\newcommand{\Ext}{\operatorname{Ext}}
\newcommand{\injdim}{\operatorname{injdim}}
\newcommand{\Proj}{\operatorname{Proj}}

\title{$\Gamma$-invariant ideals in Iwasawa algebras}
\author{K. Ardakov, S. J. Wadsley}

\address{(K. Ardakov) School of Mathematical Sciences, University of Nottingham, University Park,
Nottingham, NG7 2RD, United Kingdom}

\email{Konstantin.Ardakov@nottingham.ac.uk}

\address{(S. J. Wadsley) DPMMS, University of Cambridge, Centre for
Mathematical Sciences, Wilberforce Road, Cambridge CB3 0WB, UK}

\email{S.J.Wadsley@dpmms.cam.ac.uk}
\begin{abstract} Let $kG$ be the completed group algebra of a uniform pro-$p$ group $G$ with coefficients in a field $k$ of characteristic $p$. We study right ideals $I$ in $kG$ that are invariant under the action of another uniform pro-$p$ group $\Gamma$. We prove that if $I$ is non-zero then an irreducible component of the characteristic support of $kG/I$ must be contained in a certain finite union of rational linear subspaces of $\Spec \gr kG$. The minimal codimension of these subspaces gives a lower bound on the homological height of $I$ in terms of the action of a certain Lie algebra on $G/G^p$. If we take $\Gamma$ to be $G$ acting on itself by conjugation, then $\Gamma$-invariant right ideals of $kG$ are precisely the two-sided ideals of $kG$, and we obtain a non-trivial lower bound on the homological height of a possible non-zero two-sided ideal. For example, when $G$ is open in $\SL_n(\Zp)$ this lower bound equals $2n - 2$. This gives a significant improvement of the results of Ardakov, Wei and Zhang \cite{AWZ1} on reflexive ideals in Iwasawa algebras.
\end{abstract}
\subjclass[2000]{16L30, 16P40}

\maketitle
\let\le=\leqslant  \let\leq=\leqslant
\let\ge=\geqslant  \let\geq=\geqslant

\section{Introduction}
\subsection{Prime ideals in Iwasawa algebras}
In recent years, several attempts have been made to understand the structure of prime ideals in non-commutative Iwasawa algebras. These are the completed group algebras $\Omega_G$ of compact $p$-adic analytic groups $G$ with coefficients in the finite field $\Fp$; we refer the reader to the survey paper \cite{AB} for definitions and more details about these algebras. 

Perhaps the first result in this direction was obtained by Venjakob in \cite{V} with the classification of all prime ideals in $\Omega_G$ where $G$ is a non-abelian soluble pro-$p$ group of rank two. After this, the second author showed in \cite{Wad} that if $G$ is a Heisenberg pro-$p$ group with centre $Z$ then every non-zero prime ideal in $\Omega_G$ must contain the kernel of the map $\Omega_G\rightarrow \Omega_{G/Z}$. More significantly in \cite{AWZ1} and \cite{AWZ2} Wei, Zhang and the first author showed that there are no reflexive ideals in $\Omega_G$ if $G$ is a uniform pro-$p$ group of Chevalley type. The methods of this paper are very much in the spirit of the latter works. 

\subsection{$\Gamma$-invariant right ideals} Our basic set-up is that we have a uniform pro-$p$ group $G$ and a group $\Gamma$ acting on $G$ by group automorphisms. This action will induce an action of $\Gamma$ on $\Omega_G$ and we are interested in studying the $\Gamma$-invariant right ideals in $\Omega_G$. When $\Gamma=G$ and the action is the natural conjugation action then of course the $\Gamma$-invariant right ideals of $\Omega_G$ are just ordinary two-sided ideals. 

As in \cite{AWZ1} and \cite{AWZ2} the basic strategy is to show that under certain conditions, a $\Gamma$-invariant right ideal is controlled by the subring $\Omega_{G_2}$, where $G_2$ is the second term of the lower $p$-series of $G$. Since $G_2$ is a characterisic subgroup, the action of $\Gamma$ on $G$ restricts to an action on $G_2$ and the intersection of our ideal with $\Omega_{G_2}$ is still $\Gamma$-invariant. We then aim to show inductively that any ideal satisfying our conditions is controlled by $\Omega_{G_n}$ for every $n$. In this way we deduce that our original ideal is 0. 

In order to prove a control theorem of this type we associate to each right ideal $I$ in $\Omega_G$ a ``failure of control module" $F_I:=I/(I\cap \Omega_{G_2})\Omega_G$. Then $I$ is controlled by $\Omega_{G_2}$ precisely if $F_I=0$. 

\subsection{Microlocalisation}
Whenever we have a complete filtered ring $A$ whose associated graded ring $B$ is commutative noetherian, for each multiplicatively closed subset $T$ of $B$ consisting of homogeneous elements we can define the microlocalisation of $A$ at $T$, which is a certain completion of an Ore localisation of $A$. Heuristically, this gives rise to a `sheaf' $\mathcal{O}_A$ of complete filtered non-commutative rings on  $\Spec_{\gr}(B)$ --- the set of homogeneous prime ideals in $B$ --- viewed as a subspace of $\Spec(B)$ with its usual topology.  

Similarly, for each finitely generated $A$-module $M$ we can define a `sheaf' $\mathcal{M}$ of $\mathcal{O}_A$-modules by microlocalisation in an analogous way. As usual we say such a sheaf is supported on a subset $C$ if each stalk $\mathcal{M}_P=0$ for each $P$ not in $C$. We call the minimal such set $\Ch(M)$, the \emph{characteristic support} of $M$; see $\S$\ref{Ch}.  

\subsection{Main results}
\label{Sketch}Now the associated graded ring of $A = \Omega_G$ with respect to its natural filtration may be naturally identified with the symmetric algebra $B=\Sym(V)$, where $V=G/G_2$. To study $F_I$ we consider it as the global sections of the `sheaf' $\mathcal{F}_I$; we show in one of our main results, Theorem \ref{SuppFail}, that if $I$ is $\Gamma$-invariant then $\mathcal{F}_I$ must be supported on the set 
\[\mathcal{X}:=\{P\in\Spec_{\gr}(B) : \mathfrak{g}.v\subseteq P\mbox{ for some }v\in V\backslash 0\}.\]
Here $\mathfrak{g}$ is a certain Lie algebra that acts naturally on $V$, constructed from the action of $\Gamma$ on $G$ --- see $\S \ref{Normaliser}$ for details. Since in this case $B$ is a polynomial algebra and since $V$ is a finite set, $\mathcal{X}$ is contained in the union of finitely many rational (that is, defined over $\Fp$) linear subspaces of the affine space $\Spec(B)$.

This result puts a restriction on the characteristic support of $F_I$: the codimension of $\Ch(F_I)$, which coincides with the grade of the module $F_I$, is bounded below by the minimal dimension of a $\mathfrak{g}$-orbit of a non-zero element of $V$.

Suppose now that in addition $I$ is a prime ideal of $A$ which is not controlled by $\Omega_{G_2}$. Then $F_I$ has the same dimension as $A/I$ by Proposition \ref{Pure} and moreover $\Ch(A/I)$ is geometrically pure by Gabber's purity theorem. Since $\Ch(F_I)$ is always contained in $\Ch(A/I)$ by Proposition \ref{ChFrob}, it follows that an irreducible component of $\Ch(A/I)$ must be contained in one of the aforementioned rational linear subspaces. This gives us a severe restriction on the possible characteristic support of $A/I$. We can then deduce in Theorem \ref{Main} that the grade of $A/I$, also known as the homological height of $I$, is bounded below by \[ u:= \min \{\dim \mathfrak{g}\cdot v : v \in V \backslash 0\}.\]
In fact, the conclusion of Theorem \ref{Main} holds if we only assume that $I$ is non-zero and $\Gamma$-invariant.
\subsection{Consequences for Iwasawa algebras}
\label{Table}The effect of this for the Iwasawa algebras of Chevalley type discussed in \cite{AWZ2} is that not only are there no non-zero reflexive prime ideals in $\Omega_G$, (that is, ideals of homological height 1), but in fact non-zero prime ideals of homological height strictly less than $u$ cannot exist: see Theorem \ref{IwaChev}. We compute the lower bound $u$ in this case in $\S\ref{ChevComp}$; here is a table of the values it takes for each root system $\Phi$:
\[\begin{array}{c|cccccccccr} \Phi & A_n & B_n & C_n & D_n & E_6 & E_7 & E_8 & F_4 & G_2 \\ \hline \dim G & n^2+2n & 2n^2 + n & 2n^2 + n & 2n^2 - n & 78 & 133& 248& 52& 14 \\ u & 2n & 4n - 4 & 2n & 4n - 6 & 22 & 34 & 58 & 16 & 6  \end{array}\]
We include the dimensions of the associated Chevalley groups $G$ in this table, for the convenience of the reader.

\subsection{An outline of the paper}
This paper should be viewed as an appropriate generalisation and strengthening of the method introduced in \cite{AWZ1} and \cite{AWZ2}. In Section \ref{Prelim}, we recall basic facts about the Frobenius pairs framework developed in \cite{AWZ1}, and discuss the notions of microlocalisation and characteristic support of modules. We also prove some `geometric' properties of Frobenius pairs in $\S \ref{ChFrob}$.

The goal of Section \ref{ControlSec} is to establish Theorem \ref{Control}, a control theorem for invariant ideals; this result should be viewed as a generalisation of the control theorem for normal elements, \cite[Theorem 3.1]{AWZ1}. We introduce a new notion of a `source of derivations' $\mathcal{S}$ in $\S \ref{SourceDers}$ and an appropriate notion of $\mathcal{S}$-closure $J^\mathcal{S}$ of a graded ideal $J$ in $\S\ref{Star}$; with these in place, the analogue of the derivation hypothesis of \cite[\S 3.5]{AWZ1} simply becomes $\mathcal{D}(J^\mathcal{S}) \subseteq J$. 

In Section \ref{IwaAlgs}, we apply Theorem \ref{Control} to microlocalisations of Iwasawa algebras and prove our main result, Theorem \ref{SuppFail}. The key step in the proof is Proposition \ref{DerHyp}, which shows that the `derivation hypothesis' holds for the microlocalisation of the Iwasawa algebra at any graded prime ideal $P$ not in $\mathcal{X}$. This step relies heavily on the linear algebra calculations performed in \cite[\S1]{AWZ2}. The geometric properties of Frobenius pairs established in $\S\ref{Prelim}$ can now be used to prove the lower bound Theorem \ref{Main}, essentially in the way sketched above.

We discuss the two obvious applications of our methods in Section \ref{Appl}, and compute the invariant $u$ for uniform pro-$p$ groups of Chevalley type in $\S\ref{ChevComp}$.

\subsection{Acknowledgements}
The first author was supported by an Early Career Fellowship from the Leverhulme Trust. The second author was supported by EPSRC research grant EP/C527348/1. 

\section{Microlocalisation, characteristic support and purity}
\label{Prelim}\subsection{Frobenius pairs} Throughout, $k$ will denote an arbitrary base field of characteristic $p$. Let $B$ be a commutative $k$-algebra. The Frobenius map $x \mapsto x^p$ is a ring endomorphism 
of $B$ and gives a ring isomorphism of $B$ onto its image
\[B^{[p]} := \{b^p : b \in B\} \]
in $B$ provided that $B$ is reduced. 

Let $t$ be a positive integer. Whenever $\{y_1,\ldots,y_t\}$ is a
$t$-tuple of elements of $B$ and $\alpha =
(\alpha_1,\ldots,\alpha_t)$ is a $t$-tuple of nonnegative integers,
we define
\[{\mathbf{y}}^{\alpha} =y_1^{\alpha_1}\cdots y_t^{\alpha_t}.\]
Let $[p-1]$ denote the set $\{0,1,\ldots,p-1\}$ and let $[p-1]^t$ be
the product of $t$ copies of $[p-1]$.

\begin{defn}\cite[Definition 2.2]{AWZ1}
Let $A$ be a complete filtered $k$-algebra and let $A_1$ be
a subalgebra of $A$. We always view $A_1$ as a filtered subalgebra
of $A$, equipped with the subspace filtration $F_nA_1 := F_nA \cap
A_1$. We say that $(A,A_1)$ is a \emph{Frobenius pair} if the
following axioms are satisfied:
\begin{enumerate}[{(} i {)}]
\item 
$A_1$ is closed in $A$,
\item 
$\gr A$ is a commutative noetherian domain, and we write 
$B=\gr A$,
\item 
the image $B_1$ of $\gr A_1$ in $B$
satisfies $B^{[p]} \subseteq B_1$, and
\item 
there exist homogeneous elements $y_1,\ldots,y_t \in B$ 
such that
\[B = \bigoplus_{\alpha \in [p-1]^t}  
B_1 \mathbf{y}^{\alpha}.\]
\end{enumerate}
\end{defn}
\subsection{Microlocalisation}
\label{MicLoc}
As in \cite{AWZ1}, one of our main tools is \emph{microlocalisation}. We refer the reader to \cite[\S 4]{AWZ1} for the necessary background.

\begin{lem} Let $(A,A_1)$ be a Frobenius pair, let $T$ be a homogeneous Ore set in $B$ and let $T_1 := T \cap B_1$. Then 
\begin{enumerate}[{(}a{)}]
\item $(Q_T(A), Q_{T_1}(A_1))$ is also a Frobenius pair, and
\item $Q_T(A) = Q_{T_1}(A)$.
\end{enumerate}
\end{lem}
\begin{proof} Recall \cite[\S 4.2]{AWZ1} that the microlocalisation $Q_T(A)$ is the completion of the localisation $A_S$ of $A$ at the Ore set $S$ in $A$ defined as follows:
\[S = \{s \in A : \gr s \in T\}.\]
The microlocalisation $Q_{T_1}(A_1)$ is defined similarly. Because the proof of part (a) is very similar to that of \cite[Proposition 5.1(a)]{AWZ1}, we will omit the details. For part (b), note that by \cite[Lemma 2.3]{AWZ1} we can find a finite set $X$ that simultaneously generates $A$ as a (free) $A_1$-module, and also $Q_T(A)$ as a (free) $Q_{T_1}(A_1)$-module:
\[A = X \cdot A_1 \quad\mbox{and}\quad Q_T(A) = X \cdot Q_{T_1}(A_1).\]
Now $Q_{T_1}(A) = A \otimes_{A_1} Q_{T_1}(A)$ by definition of the microlocalisation of a module, so
\[Q_{T_1}(A) = X \cdot Q_{T_1}(A) = Q_T(A)\]
as required.
\end{proof}

\begin{cor} Let $M$ be a finitely generated $A$-module. Then $Q_T(M) \cong Q_{T_1}(M)$ as $Q_{T_1}(A_1)$-modules.
\end{cor}
\begin{proof} This is a direct consequence of part (b) of the lemma:
\[Q_T(M) = M \otimes_A Q_T(A) = M\otimes_A A \otimes_{A_1} Q_{T_1}(A_1) = M \otimes_{A_1} Q_{T_1}(A_1) = Q_{T_1}(M),\]
and these identifications respect the right $Q_{T_1}(A_1)$-module structures. \end{proof}

\subsection{Control of ideals}
\label{Unloc}
Quite generally, if $A_1$ is a subring of the ring $A$, and $I$ is a right ideal of $A$, we will write $I_1$ for the right ideal $I \cap A_1$ of $A_1$, and say that $I$ is \emph{controlled by $A_1$} if and only if $I = I_1 \cdot A$. Controlled ideals are in some sense ``understood", as they ``come from" the smaller ring $A_1$. 

Frequently we will be able to prove that the microlocalisation of $I$ is controlled by the microlocalisation of $A_1$. The next result tells us how to ``lift" this information back from the microlocalisation. 

\begin{lem}
Let $(A,A_1)$ be a Frobenius pair, let $T$ be a homogeneous Ore set in $B$ and let $I$ be a right ideal of $A$. Then $Q_T(I)$ is controlled by $Q_{T_1}(A_1)$ if and only if the ``failure of control" module $F := I / I_1A$ satisfies $(\gr F)_T = 0$.
\end{lem}
\begin{proof} The ideas in this proof were already present in the proof of \cite[Theorem 5.2]{AWZ1}, but we include the details for the convenience of the reader.

Write $A' = Q_T(A)$, $A'_1 = Q_{T_1}(A_1)$ and $I' = Q_T(I) = I \cdot A'$. Note that $I'$ can be identified with a right ideal of $A'$, by \cite[Lemma 4.4(c)]{AWZ1}. Since microlocalisation preserves pullbacks \cite[Lemma 4.4(e)]{AWZ1}, $Q_{T_1}(I) \cap Q_{T_1}(A) = Q_{T_1}(I \cap A)$, so 
\[ (I \cdot A'_1) \cap A'_1 = (I \cap A_1) \cdot A'_1.\]
By Corollary \ref{MicLoc}, $I' = I\cdot A'_1$, so
\[ (I' \cap A'_1)\cdot A' = (I\cap A_1)\cdot A'_1 \cdot A' = (I\cap A_1)A \cdot A'.\]
This shows that $I'$ is controlled by $A'_1$ if and only if $I\cdot A' = I_1A \cdot A'$. Since microlocalisation is exact, this is equivalent to $Q_T(F) = 0$. But $\gr Q_T(F) = (\gr F)_T$, so the 
result follows from the fact that $Q_T(A)$ is a complete filtered ring.
\end{proof}

\subsection{The characteristic support}
\label{Ch}Let $A$ be a filtered ring whose associated graded ring $B = \gr A$ is commutative noetherian. Let $\Spec_{\gr}(B)$ denote the set of all graded prime ideals of $B$.
\begin{defn} Let $M$ be a finitely generated $A$-module. The \emph{characteristic support} of $M$ is the following subset of $\Spec_{\gr}(B)$:
\[\Ch(M) := \{P \in \Spec_{\gr}(B): \Ann(\gr M) \subseteq P\}.\]
\end{defn}

\begin{lem} $\Ch(M)$ is independent of the choice of good filtration on $M$ that defines the associated graded module $\gr M$. 
\end{lem}
\begin{proof} This is well-known; see, for example, \cite[Chapter III, Lemma 4.1.9]{LV}.\end{proof}

Any graded prime ideal $P$ of $B$ gives rise to the homogeneous multiplicatively closed set $T_P$, which consists of all homogeneous elements of $B$ \emph{not} in $P$. We can then form the localisation $B_{T_P}$ of $B$ and the microlocalisation $Q_{T_P}(A)$ of $A$. By abuse of notation, we will always write $B_P := B_{T_P}$ and $A_P := Q_{T_P}(A)$ in this case. Furthermore, if $M$ is a finitely generated $A$-module, then we will write $M_P := Q_{T_P}(M)$ for the microlocalisation of $M$ at $T_P$.

\begin{prop} For any finitely generated $A$-module $M$, we have
\[\Ch(M) = \{P \in \Spec_{\gr}(B) : M_P \neq 0\}.\]
\end{prop}
\begin{proof} Let $P \in \Spec_{\gr}(B)$ and let $N = \gr M$; then $M_P = 0$ if and only if $N_P = 0$. Now $N$ is a quotient of a direct sum of copies of $B / \Ann(N)$, and $B / \Ann(N)$ is a submodule of a direct sum of copies of $N$, so $N_P = 0$ if and only if $(B / \Ann(N))_P = 0$. However the last condition is easily seen to be equivalent to $\Ann(N) \nsubseteq P$, and the result follows.
\end{proof}

\begin{cor} If $0 \to L \to M \to N \to 0$ is a short exact sequence of finitely generated $A$-modules, then $\Ch(M) = \Ch(L) \cup \Ch(N)$.
\end{cor}
\begin{proof} Microlocalisation is exact. \end{proof}

\subsection{Characteristic support and Frobenius pairs}
\label{ChFrob}
Let $(A,A_1)$ be a Frobenius pair. The inclusion $\iota : B_1 \hookrightarrow B$ induces a map $\iota_\ast : \Spec_{\gr}(B) \to \Spec_{\gr}(B_1)$, given by $\iota_\ast(P) = P\cap B_1$.

\begin{prop} The map $\iota_\ast$ is a bijection. Let $I$ be a right ideal of $A$, and let $I_1 = I \cap A_1$. Then
$\Ch(A/I) = \Ch(A/I_1A) = \iota_\ast^{-1}\left(\Ch(A_1/I_1)\right).$
\end{prop}
\begin{proof} Since $B$ is a finitely generated $B_1$-module by definition, $B$ is integral over $B_1$. Hence $\iota_\ast$ is surjective. Because $b^p \in B_1$ for all $b \in B$, we see that $\iota_\ast$ is injective. In fact, it is easy to see that the inverse of $\iota_\ast$ is given explicitly by the formula
\[ \iota_\ast^{-1}(\mathfrak{p}) = \sqrt{\mathfrak{p}B},\]
for any $\mathfrak{p} \in \Spec_{\gr}(B_1)$.

Let $P \in \Spec_{\gr}(B)$ and write $P_1 = \iota_\ast P$. Since $B_1 \backslash P_1 = (B \backslash P)\cap B_1$,  Corollary \ref{MicLoc} implies that $M_P = M_{P_1}$ for any finitely generated $A$-module $M$. We will use Proposition \ref{Ch} without further mention in what follows.

Since $A/I$ is a quotient of $A/I_1A$, $\Ch(A/I) \subseteq \Ch(A/I_1A)$ by Corollary \ref{Ch}. Now if $P_1 \notin \Ch(A_1/I_1)$, then $(A_1/I_1)_{P_1 } = 0$, so 
\[(A/I_1A)_P = (A_1/I_1)\otimes_{A_1}A\otimes_A A_P = (A_1/I_1)\otimes_{A_1} A_P = 0\]
whence $P\notin \Ch(A/I_1A)$. This shows that $\Ch(A/I_1A) \subseteq \iota_\ast^{-1}\Ch(A_1/I_1)$.

Finally, if $P \notin \Ch(A/I)$ then $(A/I)_P = 0$ and hence $(A/I)_{P_1} = 0$. Since $A_1/I_1$ is an $A_1$-submodule of $A/I$, Corollary \ref{Ch} implies that $P \notin \iota_\ast^{-1} \Ch(A_1/I_1)$. Hence $\iota_\ast^{-1}\Ch(A_1/I_1) \subseteq \Ch(A/I)$ and the proof is complete.
\end{proof}
\subsection{Purity of modules}
\label{Pure}
Let $A$ be an Auslander-Gorenstein ring and let $M$ be a finitely generated $A$-module. Recall that $M$ has a \emph{grade} $j_A(M)$ defined by the formula
\[j_A(M) = \min \{ j : \Ext_A^j(M,A)\neq 0\}.\]
Recall that if $0 \to L \to M \to N \to 0$ is a short exact sequence, then we have
\[j_A(M) = \min\{j_A(L), j_A(N)\}.\]
$M$ is said to be \emph{pure} if $j_A(N) = j_A(M)$ for all non-zero submodules $N$ of $M$; clearly any submodule of a pure module is itself pure.

\begin{lem} Let $(A,A_1)$ be a Frobenius pair such that $B$ and $B_1$ are Gorenstein, and let $M$ be a finitely generated $A$-module.
\begin{enumerate}[{(}a{)}]
\item $A$ and $A_1$ are also Auslander-Gorenstein.
\item The grade of $M$ equals the codimension of the characteristic support of $M$:
\[j_A(M) = j_B(\gr M) = \min \{\hgt P : P \in \Ch(M)\}.\]
\item $j_A(M) = j_{A_1}(M|_{A_1}).$
\item $M$ is pure if and only if $M|_{A_1}$ is pure.
\end{enumerate}
\end{lem}
\begin{proof} (a) Use  \cite[Theorem 3.9]{Bj}.

(b) For the first equality, use \cite[Remark 5.8]{BjE}. For the second equality, see \cite{Bass}.

(c) The map $\iota_\ast : \Spec_{\gr}(B) \to \Spec_{\gr}(B_1)$ extends to an order preserving bijection between the usual spectra $\Spec(B)$ and $\Spec(B_1)$; this shows that $\hgt \iota_\ast P = \hgt P$ for any $P\in\Spec_{\gr}(B)$. Since $M$ is finitely generated, $M$ is a finite extension of cyclic $A$-modules; using Corollary \ref{Ch} and Proposition \ref{ChFrob}, we see that 
\[\iota_\ast \Ch(M) = \Ch(M|_{A_1}).\] 
Part (c) now follows from part (b).

(d) $(\Leftarrow)$ This is easy, given part (c). 

$(\Rightarrow)$ Let $N$ be a non-zero $A_1$-submodule of $M$. Since $N\otimes_{A_1} A$ surjects onto the non-zero $A$-submodule $N\cdot A$ of $M$ and since $M$ is pure, we have
\[j_{A_1}(M) = j_A(M) = j_A(N\cdot A) \geq j_A(N\otimes_{A_1}A) = j_{A_1}(N) \geq j_{A_1}(M)\]
Here we have used the fact that $A$ is a free $A_1$-module \cite[Lemma 2.3]{AWZ1}. Hence $j_{A_1}(N) = j_{A_1}(M)$ and $M|_{A_1}$ is pure. 
\end{proof}

\begin{prop} Let $I$ be a right ideal of $A$, let $I_1 = I\cap A_1$ and suppose that $A/I$ is pure. Then either $I/I_1A$ is zero, or it is pure of the same grade as $A/I$.
\end{prop}
\begin{proof} By part (d) of the lemma, $A_1/I_1$ is pure, being an $A_1$-submodule of the pure $A_1$-module $A/I$. It will be enough to show that $A/I_1A \cong (A_1/I_1) \otimes_{A_1} A$ is also pure, and has the same grade as $A/I$.

Recall that by \cite[Theorem 2.12]{BjE}, a finitely generated module $M$ over an Auslander-Gorenstein ring $R$ is pure if and only if $\Ext_R^i(\Ext_R^i(M,R),R) = 0$ for all $i > j_R(M)$. Hence $\Ext_{A_1}^i(\Ext_{A_1}^i(A_1/I_1,A_1),A_1) = 0$ for all $i > j_{A_1}(A_1/I_1)$. 

Since $A$ is a free right and left $A_1$-module by \cite[Lemma 2.3]{AWZ1}, \cite[Proposition 1.2]{AWZ1} implies that $j_A(A/I_1A) = j_{A_1}(A/I_1)$ and also that 
\[\Ext_A^i(\Ext_A^i(A/I_1A, A), A) \cong \Ext_{A_1}^i(\Ext_{A_1}^i(A_1/I_1,A_1),A_1) \otimes_{A_1} A = 0\]
for all $i > j_A(A/I_1A)$. So $A/I_1A$ is pure, and our result follows.
\end{proof}

\section{A control theorem for $\mathcal{S}$-invariant right ideals}
\label{ControlSec}
\subsection{Inducing derivations on $\gr A$}
\label{Source}
Let $A$ be a filtered ring with associated graded ring $B$ and
let $\alpha$ be a ring endomorphism of $A$. Suppose that there is an integer $m_\alpha \geq 1$ such that
\[(\alpha - 1)(F_nA) \subseteq F_{n-m_\alpha}A\]
for all $n\in\mathbb{Z}$. This induces additive maps
\[\begin{array}{cccc}
d_{\alpha} :& \frac{F_nA}{F_{n-1}A} &\to & \frac{F_{n-m_\alpha}A}{F_{n-m_\alpha-1}A} \\
\quad &\quad \\
& x + F_{n-1}A &\mapsto &\alpha(x) - x + F_{n-m_\alpha-1}A
\end{array}
\]
for each $n\in\mathbb{Z}$, which patch together to give a graded endomorphism $d_\alpha$ of the abelian group $B$.

\begin{lem} $d_\alpha$ is a graded derivation of $B$ of degree $m_\alpha$.
\end{lem}
\begin{proof} Let $x\in F_mA$ and $y \in F_nA$, so that $X = x + F_{m-1}A$ and $Y = y + F_{n-1}A$ are homogeneous elements of $B$ of degree $m$ and $n$ respectively. Then
\[\begin{array}{llll}
d_\alpha(X \cdot Y) &=& \alpha(xy)-xy + F_{m+n-m_\alpha-1}A, \\
d_\alpha(X) \cdot Y &=& \alpha(x)y - xy + F_{m+n-m_\alpha-1}A, \quad\mbox{and} \\
X \cdot d_\alpha(Y) &=& x\alpha(y) - xy + F_{m+n-m_\alpha-1}A.
\end{array}\]
Because $m_\alpha \geq 1$, $(\alpha(x)-x)(\alpha(y)-y) \in F_{m+n-2m_\alpha}A \subseteq F_{m+n-m_\alpha-1}A$. Hence
\[\alpha(xy) - xy \equiv \alpha(x)y - xy + x\alpha(y) - xy \mod F_{m+n-m_\alpha-1}A\]
and the result follows.
\end{proof}

\subsection{New sources of derivations}
\label{SourceDers}We now introduce a new notion of ``source of derivations" for a Frobenius pair, which is slightly different from the one introduced in \cite[\S 3.3]{AWZ1}. We hope that the inconsistency in terminology will not cause any confusion; it is just a matter of language. After reading this paper, the reader may get the feeling that the ``real" source of our derivations --- at least for our applications --- is the Lie algebra $\mathfrak{g}$ defined below in $\S\ref{Normaliser}$.

\begin{defn}
A \emph{source of derivations} for a Frobenius pair $(A,A_1)$ is
a set $\mathbf{a} = \{\alpha_0,\alpha_1,\alpha_2,\ldots\}$ of endomorphisms of $A$ such that
there exist functions $\theta, \theta_1:\mathbf{a} \to\mathbb{N}$
satisfying the following conditions:
\begin{enumerate}[{(} i {)}]
\item
$(\alpha_r - 1)F_nA \subseteq F_{n - \theta(\alpha_r)}A$
for all $r\geq 0$ and all $n\in\mathbb{Z}$
\item
$(\alpha_r - 1)F_nA_1 \subseteq F_{n - \theta_1(\alpha_r)}A$ for all $r\geq 0$
and all $n\in\mathbb{Z}$,
\item $\theta_1(a_r) - \theta(\alpha_r) \to \infty$ as $r \to \infty$.
\end{enumerate}
\end{defn}

As in \cite[\S 5]{AWZ1}, we will need to know that sources of derivations are compatible with microlocalisations. The next result shows that this is indeed the case.

\begin{prop} Let $(A,A_1)$ be a Frobenius pair, let $T$ be a homogeneous Ore set in $B$, and let $T_1 = B_1 \cap T$. Then
\begin{enumerate}[{(}a{)}]
\item each source of derivations $\mathbf{a}$ of $(A,A_1)$ induces a source of derivations $\mathbf{a}_T$ of $(Q_T(A),Q_{T_1}(A_1))$,
\item the derivations of $B_T$ induced by $\mathbf{a}_T$ coincide with the extensions to $B_T$ of the derivations of $B$ induced by $\mathbf{a}$.
\end{enumerate}
\end{prop}
\begin{proof}
Let $\alpha$ be a ring endomorphism of $A$ such that $(\alpha - 1)(F_nA) \subseteq F_{n - m}A$ for some integer $m\geq 1$, for all $n\in\mathbb{Z}$. Let $x \in A$; since $\deg (\alpha(x) - x) < \deg x$, we have $\gr \alpha(x) = \gr x$ and in particular $\deg \alpha(x) = \deg x$. Hence $\alpha(S) \subseteq S$, where $S$ denotes the Ore subset of $A$ determined by $T$, see $\S\ref{MicLoc}$. Thus $\alpha$ extends to an endomorphism $\alpha$ of the Ore localisation $A_S$. 

Now if $r \in A$ and $s \in S$, then the formula
\[\alpha(rs^{-1}) - rs^{-1} = (\alpha(r) - r)\alpha(s)^{-1} - rs^{-1}(\alpha(s) - s)\alpha(s)^{-1}\]
together with the explicit description of the filtration on $A_S$ given in \cite[Lemma 4.2]{AWZ1} shows that 
\[(\alpha - 1)(F_nA_S) \subseteq F_{n-m}A_S\]
for all $n \in \mathbb{Z}$. Because $m \geq 1$, it follows from this that $\alpha$ preserves the filtration on $A_S$ and hence extends to a ring endomorphism $\alpha$ of the completion $Q_T(A)$ such that $(\alpha - 1)F_nQ_T(A) \subseteq F_{n-m}Q_T(A)$ for all $n \in \mathbb{Z}$. Similarly, if $(\alpha - 1)F_nA_1 \subseteq F_{n - m}A$ then $(\alpha - 1)F_nQ_T(A_1) \subseteq F_{n - m}Q_T(A)$. Part (a) follows, and part (b) is clear.
\end{proof}

\subsection{The delta function}
\label{DeltaFunction} Let $(A,A_1)$ be a Frobenius pair and $n$ be
an integer. Each filtered part $F_nA_1$ is closed in $A_1$ by
definition of the filtration topology, and $A_1$ is closed in $A$ by
assumption. Hence $F_nA_1$ is closed in $A$, which can be expressed
as follows:
\[F_nA_1 = \bigcap_{m\geq 0} \left(F_nA_1 + F_{n-m}A\right).\]

We can now define a key invariant of elements of $A \backslash A_1$:

\begin{defn}
For any $w \in A\backslash A_1$, let $n = \deg w$ and define
\[\delta(w) := \max\{m : w \in F_nA_1 + F_{n-m}A\} .\]
\end{defn}

Clearly $\delta(w)\geq 0$. Note that if $w\in F_nA \backslash A_1$,
then $w \notin F_nA_1 + F_{n-m}A$ for some $m\geq 0$ by the above
remarks, so the definition makes sense. The number $\delta(w)$ measures how closely the element $w$ can be
approximated by elements of $A_1$. It should be remarked that
$\delta(w) > 0$ if and only if $\gr w \in B_1$, since both
conditions are equivalent to $w \in F_nA_1 + F_{n-1}A$.

Now suppose that $w\in A \backslash A_1$. By the definition of $\delta$,
we can find elements $x\in F_nA_1$ and $y\in F_{n-\delta}A$ such
that $w = x + y$; if
$\delta = 0$ we take $x$ to be zero. Note that $y \notin F_{n - \delta - 1}A$
by the maximality of $\delta$ and hence
\[Y_w := \gr y = y + F_{n - \delta - 1}A.
\]
In view of our assumption on $x$, we have $Y_w = \gr w$ when $\delta
= 0$.

\subsection{$\mathcal{S}$-closures}
\label{Star}
Let $(A,A_1)$ be a Frobenius pair and let $\mathcal{S}$ be a fixed set of sources of derivations of $A$. If $I$ is a right ideal of $A$, we say that $I$ is \emph{$\mathcal{S}$-invariant} if for all $\mathbf{a}\in\mathcal{S}$, $\alpha_r(I) \subseteq I$ for all $r \gg 0$.

\begin{defn} Let $(A,A_1)$ be a Frobenius pair, let $\mathcal{S}$ be a set of sources of derivations of $A$ and let $J$ be a graded ideal of $B$. The \emph{$\mathcal{S}$-closure} $J^\mathcal{S}$ of $J$ is defined to be
\[J^{\mathcal{S}} := \{ Y \in B : \quad \forall \mathbf{a} \in \mathcal{S}, \quad d_{\alpha_r}(Y) \in J \quad\mbox{for all}\quad r \gg 0.\}\]
\end{defn}

Because $d_{\alpha_r}$ is a $B_1$-linear derivation of $B$ for large enough $r$, we see that $J^{\mathcal{S}}$ is a $B_1$-submodule of $B$ containing $B_1$. It is in fact a graded $B_1$-submodule.

\begin{prop} Let $I$ be an $\mathcal{S}$-invariant right ideal of $A$. Then for any $w \in I \backslash A_1$, $Y_w \in J^{\mathcal{S}}$.
\end{prop}
\begin{proof} Let us write $w = x + y$ as in the previous subsection and let $\mathbf{a} \in \mathcal{S}$. We can find an integer $r_0\geq 1$ such that $\theta_1(\alpha_r) - \theta(\alpha_r) > \delta := \delta(w)$
for all $r \geq r_0$. Therefore
\[\begin{aligned}
\alpha_r(x) - x &\in F_{n - \theta_1(\alpha_r)}A \subseteq
F_{n - \delta - \theta(\alpha_r) - 1}A \quad\mbox{and} \\
\alpha_r(y) - y \; &\in F_{n - \delta - \theta(\alpha_r)}A,
\end{aligned}
\]
for all $r \geq r_0$. Hence
$$\begin{aligned}
\alpha_r(w) - w \; &\in F_{n - \delta - \theta(\alpha_r)}A,
\quad \mbox{and} \\
\alpha_r(w) - w \; & \equiv \alpha_r(y) - y \mod
F_{n - \delta - \theta(\alpha_r) - 1}A
\end{aligned}
$$
for all $r\geq r_0$. We can rewrite the above as follows:
\[\alpha_r(w) - w  + F_{n-\delta-\theta(\alpha_r) - 1}A =
\alpha_r(y)-y + F_{n-\delta-\theta(\alpha_r) - 1}A
= d_{\alpha_r}(Y_w)\]
for $r \geq r_0$. Since $w \in I$ and $I$ is $\mathcal{S}$-invariant,
$d_{\alpha_r}(Y_w)$ must lie in the ideal $J = \gr I$ of $B$ for $r \gg 0$, and
hence $Y_w \in J^{\mathcal{S}}$ as required.
\end{proof}

Let $\mathcal{D}$ denote the set of all $B_1$-linear derivations of $B$.

\begin{cor}
Suppose that $\mathcal{D}(J^{\mathcal{S}}) \subseteq J$. Then $J$ is controlled by $B_1$: $J = (J\cap B_1)B$.
\end{cor}
\begin{proof} By \cite[Proposition 2.4(d)]{AWZ1}, it is enough to show that $\mathcal{D}(J) \subseteq J$. So let $X \in J$ be a homogeneous element. If $X \in B_1$ then $\mathcal{D}(X) = 0 \in J$, so assume $X \notin B_1$. Choose $w \in I$ such that $X = \gr w$; then $\delta(w) = 0$ by definition and $X = Y_w$. Hence $X \in J^{\mathcal{S}}$ by the Proposition and hence $\mathcal{D}(X) \subseteq J$ by the assumption on $J$. The result follows.
\end{proof}

\subsection{The control theorem}
\label{Control}
We can now state and prove our first main result, a control theorem for $\mathcal{S}$-invariant ideals. It should be viewed as a generalization of the control theorem for normal elements \cite[Theorem 3.1]{AWZ1}.
\begin{thm}
Let $(A,A_1)$ be a Frobenius pair, let $\mathcal{S}$ be a set of sources of derivations, let $I$ be a $\mathcal{S}$-invariant right ideal of $A$ and let $J := \gr I$. If $\mathcal{D}(J^{\mathcal{S}}) \subseteq J$ then $I$ is controlled by $A_1$:
\[I = (I \cap A_1)\cdot A.\]
\end{thm}
\begin{proof}
We will first show that $J\cap B_1 \subseteq \gr (I \cap A_1)$. Let $X \in J \cap B_1$ be homogeneous of degree $n$ say, and choose $w \in I$ such that $\gr w = X$. If $w \in A_1$ then $X = \gr w \in \gr (I \cap A_1)$ as required, so assume that $w \notin A_1$. Write $w = x + y$ as in $\S\ref{DeltaFunction}$;  by Proposition \ref{Star}, $Y := Y_w = \gr y \in J^{\mathcal{S}}$, so $\mathcal{D}(Y) \subseteq J$ by assumption on $J$.

Since $J$ is controlled by $B_1$ by Corollary \ref{Star}, applying \cite[Proposition 2.4(c)]{AWZ1} to the image of $Y$ in $B/J$ shows that $Y \in J + B_1$. 

Write $\delta = \delta(w)$, so that $\deg y = n - \delta$. Note that $\delta > 0$ because $X \in B_1$. We can find some 
$s \in I\cap F_{n-\delta}A$, $z \in F_{n-\delta}A_1$ and $\epsilon \in F_{n-\delta - 1}A$ such that $y = s + z + \epsilon$. Then $w' := w - s \in I$ and $\gr w' = \gr w = X$ because $s \in F_{n-\delta}A$ and $\delta > 0$. Moreover, $w' = x + z + \epsilon \in F_nA_1 + F_{n - \delta - 1}A$, so that $\delta(w') > \delta(w)$.

Iterating the above argument, we can construct a sequence $w_1, w_2, w_3, \ldots$ of elements of $I$ having the following properties:
\begin{itemize}
\item $\gr w_i = X$,
\item $w_i \notin A_1$,
\item $w_{i+1} \equiv w_i \mod F_{n - \delta(w_i)}A$, and 
\item $\delta(w_{i+1}) > \delta(w_i)$
\end{itemize}
for all $i\geq 1$. Note that we may always assume that $w_i \notin A_1$, because if any $w_i$ does happen to lie in $A_1$ then $X = \gr w_i \in \gr (I \cap A_1)$ and we're done.

This sequence converges to an element $u \in A$ such that $\gr u = X$. Since the filtration on $A$ is complete and since $B = \gr A$ is noetherian, $I$ is closed in the filtration topology by \cite[Chapter II, Theorem 2.1.2(6)]{LV} so $u \in I$. Given an integer $m > 0$, $\delta(w_i) > m$ and $u - w_i \in F_{n-m}A$ for sufficiently large $i$, so 
\[u  = w_i + (u-w_i) \in (F_nA_1 + F_{n-\delta(w_i)}A) + F_{n-m}A \subseteq F_nA_1 + F_{n - m}A\]
for all $m>0$. Because $A_1$ is closed in $A$, $u \in A_1$ and therefore $X = \gr u \in \gr(I \cap A_1)$. 

Thus $J\cap B_1 \subseteq \gr (I \cap A_1)$ as claimed. Because $\gr (I\cap A_1)$ is obviously contained in $J \cap B_1$, we have the equality $J \cap B_1 = \gr (I\cap A_1)$. Now 
\[\gr((I\cap A_1)A) = \gr(I\cap A_1)\cdot \gr A = (J\cap B_1)\cdot B = J = \gr I\]
and therefore $I = (I\cap A_1)A$ by \cite[Chapter II, Lemma 1.2.9]{LV}. 
\end{proof}

\section{Iwasawa algebras}
\label{IwaAlgs}
\subsection{Uniform $\Gamma$-actions}
\label{UnifAct}
Let $p$ be an odd prime and let $\Gamma$ and $G$ be uniform pro-$p$ groups. We assume that $\Gamma$ acts on $G$ by group automorphisms and that the action is \emph{uniform}:
\[\gamma\cdot g \equiv g \mod G^p\]
for all $\gamma \in \Gamma$ and $g \in G$. Let $\tau : \Gamma \to \Aut(G)$ be the associated group homomorphism. Let $L_G$ denote the $\Zp$-Lie algebra of $G$ --- this is a free $\Zp$-module of rank $d = \dim G$. Any automorphism of $G$ gives rise to an automorphism of $L_G$: this gives rise to a natural injection 
\[\iota : \Aut(G) \hookrightarrow \GL(L_G).\]
Clearly $\Gamma$ acts uniformly on $G$ if and only if the image of $\iota\tau$ is contained in the first congruence subgroup $\Gamma_1(\GL(L_G)) := \ker(\GL(L_G) \to \GL(L_G/pL_G))$ of $\GL(L_G)$. Since $\Gamma_1(\GL(L_G))$ has finite index in $\GL(L_G)$, we see that if $\Gamma$ is any pro-$p$ group of finite rank acting on $G$ by group automorphisms, then $\Gamma$ always has a uniform pro-$p$ subgroup $\Gamma_1$ of finite index that acts uniformly. 

\subsection{Some Lie theory} \label{Normaliser} The category of uniform pro-$p$ groups is isomorphic to the category of powerful Lie algebras by \cite[Theorem 9.10]{DDMS}, so the homomorphism $\iota \tau$ gives rise to a Lie algebra homomorphism
\[\sigma = \log \circ \iota\tau \circ \exp : L_\Gamma \to p\End_{\Zp}(L_G)\]
since $p\End_{\Zp}(L_G)$ is the $\Zp$-Lie algebra of $\Gamma_1(\GL(L_G))$. In other words, $L_G$ is naturally a $L_\Gamma$-module, acting by derivations and moreover
\[x\cdot L_G \subseteq pL_G \quad\mbox{for all}\quad x \in L_\Gamma.\]
Let $N_\Gamma = \{x \in \Qp L_\Gamma : x\cdot L_G \subseteq L_G\}$ be the inverse image of $\End_{\Zp}(L_G)$ under the homomorphism
\[\sigma: \Qp L_\Gamma \to \End_{\Qp}(\Qp L_G).\]
Note that $\frac{1}{p}L_\Gamma \subseteq N_\Gamma$. $N_\Gamma$ also contains $\ker \sigma$ and $N_\Gamma / \ker \sigma$ is a finitely generated $\Zp$-module. Hence 
\[\mathfrak{g} := N_\Gamma/pN_\Gamma\]
is a finite dimensional $\Fp$-Lie algebra. Define
\[V := L_G/pL_G,\]
an $\Fp$-vector space of dimension $d$. Letting $\overline{-} : L_G \twoheadrightarrow V$ and $\overline{-}:N_\Gamma \twoheadrightarrow \mathfrak{g}$ denote the natural surjections, $V$ becomes a $\mathfrak{g}$-module via the rule
\[\overline{x}\cdot\overline{y} = \overline{x\cdot y}\]
for all $x\in N_\Gamma$ and $y\in L_G$. Let $\rho : \mathfrak{g} \to \End (V)$ be the associated homomorphism.

\subsection{The derivation hypothesis} 
\label{DerHyp}
As explained in \cite[\S 1.3]{AWZ2}, every endomorphism $\varphi$ of $V$ extends to a derivation of $\Sym(V)$ and for each $r\geq 0$ we also have the deformed derivations $\varphi^{[p^r]}$ of $\Sym(V)$, defined by the rule
\[\varphi^{[p^r]}(v) = \varphi(v)^{p^r} \quad\mbox{for all}\quad v \in V.\]
Let $B = \Sym(V\otimes k)$ and let $P$ be a graded prime ideal of $B$. Let $\mathcal{D} = \Der_k(B)$ denote the space of all $k$-linear derivations of $B$; then $\mathcal{D}$ is canonically isomorphic to $B \otimes_{\Fp} V^\ast$ and $\Der_k(B_P)$ is canonically isomorphic to $B_P \otimes_{\Fp} V^\ast = \mathcal{D}_P$ (see $\S\ref{Ch}$ for the notation). The derivations $\varphi^{[p^r]}$ extend to $B_P$, so we can also think of them as lying in $\mathcal{D}_P$.

We can view $V$ as a $\mathfrak{g}$-module via $\rho$; in this way, $V^\ast$ is also naturally a $\mathfrak{g}$-module. The next result gives a sufficient condition that ensures that a local analogue of the derivation hypothesis of \cite[\S 3.5]{AWZ1} holds.

\begin{prop} Let $P$ be a graded prime ideal of $B$ which does not contain $\mathfrak{g}.v$ for any $v \in V \backslash 0$. Let $J$ be a graded ideal of $B_P$ and let $Y \in B_P$ be such that for all $x \in \mathfrak{g}$, we have
\[ \rho(x)^{[p^r]}(Y) \in J \quad\mbox{for all}\quad  r \gg 0.\]
Then $\mathcal{D}_P(Y) \subseteq J$.
\end{prop}
\begin{proof}
Fix $x \in \mathfrak{g}$ and $f \in V^\ast$. By \cite[Proposition 1.4]{AWZ2}, some $B$-linear combination of the derivations $\rho(x)^{[p^r]}$ equals $U (x\cdot f)$, where $U$ is some product of elements $u$, each lying in $x.V \backslash \ker f$. If $f(P\cap V) = 0$ then every such $u$ lies outside of $P$ and is hence a unit in $B_P$. Let $W = (P\cap V)^\perp$ be the annihilator of $P\cap V$ in $V^\ast$; then
\[(x\cdot f)(Y) \in J \quad\mbox{for all}\quad x \in \mathfrak{g}\quad\mbox{and all}\quad f \in W.\] 
Now if $\mathfrak{g} \cdot W  < V^\ast$ then there exists $v \in V\backslash 0$ such that $(\mathfrak{g} \cdot W)(v) = 0$. But then $W(\mathfrak{g}.v) = 0$, which forces $\mathfrak{g}.v \subseteq P \cap V$ and hence contradicts our assumption on $P$. 

Hence $\mathfrak{g} \cdot W = V^\ast$; however $\mathcal{D}_P$ is generated by $V^\ast$ as a $B_P$-module, so $\mathcal{D}_P(Y) \subseteq J$ as required.
\end{proof}

Note that for ``most" $P$, the intersection $P \cap V$ will be zero and then $P$ can only contain $\mathfrak{g}.v$ if $v$ lies in the space of $\mathfrak{g}$-invariants $V^{\mathfrak{g}}$ of $V$. Since $V^{\mathfrak{g}}$ can be arranged to be zero in many interesting cases, this means that the condition on $P$ imposed above is not very strong. 

\subsection{Derivations for Iwasawa algebras}\label{DersIwa}
Let $A = kG$ and $A_1 = kG^p$ be the completed group algebras of $G$ and $G^p$, with coefficients in our ground field $k$. By \cite[Theorem 6.6]{AWZ1}, $(A,A_1)$ is a Frobenius pair, and by \cite[Lemma 6.2(d) and Theorem 6.4]{AWZ1}, there is a canonical isomorphism
\[\Sym(V\otimes_{\Fp} k) \stackrel{\cong}{\longrightarrow} \gr A.\]
Compare the following result with \cite[Proposition 3.3]{AWZ2}.
\begin{prop} Let $x \in \mathfrak{g}$ be non-zero, and choose a lift $\tilde{x}$ of $x$ in $N_\Gamma\backslash pN_\Gamma$. Let $m\geq 1$ be such that $p^m\tilde{x} \in L_\Gamma$. Let $\alpha = \tau(\exp(p^m\tilde{x})) \in \Aut(G)$ and view $\alpha$ as an algebra endomorphism of $A = kG$. Then
\begin{enumerate}[{(}a{)}]
\item $(\alpha - 1)F_nA \subseteq F_{n - p^m + 1}A,$ for all $n \in \mathbb{Z}$,
\item $(\alpha - 1)F_nA_1 \subseteq F_{n - p^{m+1} + p}A$ for all $n \in \mathbb{Z}$, and
\item $d_\alpha = \rho(x)^{[p^m]}$ as derivations of $\gr A$.
\end{enumerate}
\end{prop}
\begin{proof} Let $C$ be the procyclic subgroup of $\Gamma$ generated by $\gamma := \exp(p^m\tilde{x}) \in \Gamma$. 
Let $H$ be the semidirect product of $G$ and $C$ with $\gamma$ acting on $G$ by the automorphism $\alpha$. So inside this new group $H$, we have the relation
\[ \gamma g \gamma^{-1} = \gamma \cdot g\]
for all $g \in G$. The group $H$ is uniform and $L_H$ is the semidirect product of $L_G$ and $L_C = p^m\tilde{x}\Zp$, with $p^m\tilde{x}$ acting via the derivation $\sigma(p^m\tilde{x}) : L_G \to L_G$:
\[L_H = L_G \rtimes L_C.\]
Because $L_C$ is abelian and $\tilde{x} \cdot L_G \subseteq L_G$, $[p^m\tilde{x}, L_H] \subseteq p^m L_H$. Hence by \cite[Proposition 6.7]{AWZ1}, the following relations hold in $kH$ for all $n \in \mathbb{Z}$:
\[\begin{array}{ccc} [\gamma, F_nkH] &\subseteq &F_{n - p^m + 1}kH \\

[\gamma, F_nkH^p] &\subseteq &F_{n - p^{m+1} + p}kH.\end{array}\]
Now $(\alpha - 1)(b) = \gamma b \gamma^{-1} - b = [\gamma,b]\gamma^{-1}$ for all $b \in kG$ and
\[F_nkH \cap kG = F_nkG\]
for all $n\in\mathbb{Z}$. Parts (a) and (b) follow.

Finally, part (c) follows from \cite[Theorem 6.8]{AWZ1}: one only needs to note that $d_\alpha$ coincides with the restriction to $\gr kG$ of the derivation $\{\gamma,-\}_{p^m-1}$ of $\gr kH$. 
\end{proof}

\subsection{The support of the ``failure of control" module}
\label{SuppFail}We can now put the main pieces together and prove a refined version of \cite[Theorem 5.2]{AWZ1}. The theorem below places a severe restriction on the characteristic support of the failure of control module of any $\Gamma$-invariant right ideal of $kG$; see $\S\ref{Ch}$ for the notation.

\begin{thm} Let $I$ be a $\Gamma$-invariant right ideal of $kG$ and let $F = I / (I \cap kG^p)kG$ be the failure of control module. Then for any $P \in \Ch(F)$ there exists $v \in V\backslash 0$ such that $\mathfrak{g}.v\subseteq P$.
\end{thm}
\begin{proof} Suppose for a contradiction that $P$ does not contain any subspace of $V$ of the form $\mathfrak{g}.v$ for $v \in V\backslash 0$. Let $P_1 = \iota_\ast P$; by Lemma \ref{MicLoc}, $(A_P, (A_1)_{P_1})$ is a Frobenius pair and we plan to apply the Control Theorem, Theorem \ref{Control} to it.

Let $x\in \mathfrak{g}$ and choose a lift $\tilde{x} \in N_{\Gamma}$ for $x$. There exists $m_x \geq 1$ such that $p^{m_x}\tilde{x} \in L_\Gamma$; then $\gamma_x := \exp(p^{m_x}\tilde{x})$ lies in $\Gamma$, so
\[\mathbf{a}(x) = \{\tau(\gamma_x), \tau(\gamma_x)^p, \tau(\gamma_x)^{p^2}, \ldots\}\]
is a source of derivations of $(A,A_1)$ by Proposition \ref{DersIwa}(a) and (b). We're only interested in the derivations of $B$ induced by $\mathbf{a}(x)$; by Proposition \ref{DersIwa}(c) these derivations do not depend on the choice of $\tilde{x}$, being precisely the $\rho(x)^{[p^r]}$ for $r \geq m_x$.

Let $\mathcal{S} := \{ \mathbf{a}(x) : x \in \mathfrak{g}\}$; then $\mathcal{S}_P := \{\mathbf{a}(x)_{T_P} : x \in \mathfrak{g}\}$ is a set of sources of derivations of $(A_P, (A_1)_{P_1})$ by Proposition \ref{SourceDers}(a).

Since $I$ is $\Gamma$-invariant, $I$ is clearly $\mathcal{S}$-invariant in the sense of $\S\ref{Star}$, and the 	definition of $\mathbf{a}(x)_{T_P}$ shows that the microlocalised right ideal $I_P$ of $A_P$ is $\mathcal{S}_P$-invariant. 

Let $J = \gr I_P$. In view of Proposition \ref{DersIwa}(c) and Proposition \ref{SourceDers}(b), if $Y \in B_P$ lies in the $\mathcal{S}_P$-closure $J^{\mathcal{S}_P}$ of $J$, then for all $x \in \mathfrak{g}$,
\[\rho(x)^{[p^r]}(Y) \in J\quad\mbox{for all}\quad r \gg 0.\]
It now follows from Proposition \ref{DerHyp} that $\mathcal{D}_P(Y)$ must be contained in $J$, and hence all the conditions of Theorem \ref{Control} are satisfied. We can therefore deduce from Theorem \ref{Control} that $I_P$ is controlled by $(A_1)_{P_1}$. Now Lemma \ref{Unloc} implies that $(\gr F)_P = 0$ and hence $P \notin \Ch(F)$, a contradiction.
\end{proof}

\subsection{Another control theorem}
\label{ContKG}Recall the definition of purity from $\S\ref{Pure}$.
\begin{thm} Let $I$ be a $\Gamma$-invariant right ideal of $kG$ and suppose that $kG/I$ is pure. Suppose that no  minimal prime $P$ above $\gr I$ contains $\mathfrak{g}.v$ for any $v\in V\backslash 0$. Then $I$ is controlled by $kG^p$.
\end{thm}
\begin{proof} Let $F = I / I_1A$ be the failure of control module, and let $\mathcal{X}$ denote the set of all graded prime ideals of $B$ that contain some $\mathfrak{g}.v$ for $v\in V\backslash 0$. By Corollary \ref{Ch}, Proposition \ref{ChFrob} and Theorem \ref{SuppFail},
\[\Ch(F) \subseteq \Ch(A/I_1A) \cap \mathcal{X} = \Ch(A/I) \cap \mathcal{X}.\]
Since $A/I$ is pure,  Gabber's Purity of the Characteristic Variety theorem \cite[Corollary 5.21]{BjE} implies that every prime in $\min \Ch(A/I)$ has the same height. Because none of these primes lie in $\mathcal{X}$ by assumption, it now follows from Lemma \ref{Pure}(b) that $j(F) > j(A/I)$. Therefore $F = 0$ by Proposition \ref{Pure}, as required.
\end{proof}

\subsection{Consequences for $\Ch(A/I)$ when $A/I$ is pure}\label{Restrictions}  
The group $\Gamma$ acts naturally on $G^p$, so $L_\Gamma$ acts on $L_{G^p} = pL_G$. It is easy to see that the normaliser $N_\Gamma$ for this action is the same as before, meaning that $\mathfrak{g}$ is unchanged. Now $\mathfrak{g}$ also acts on $V_1 := L_{G^p}/L_{G^{p^2}} = pL_G / p^2L_G$. Clearly the map $x \mapsto px$ induces an isomorphism of $\mathfrak{g}$-modules between $V$ and $V_1$; if we view $V$ and $V_1$ as being embedded into $B$ and $B_1$ respectively, then this isomorphism is given by $v \mapsto v^p$ for any $v \in V$.

\begin{lem} Let $\mathcal{X}_1 = \{\mathfrak{p} \in \Spec_{\gr}(B_1) : \mathfrak{p} \supseteq \mathfrak{g}\cdot v_1$ for some $v_1 \in V_1\backslash 0\}$. Then 
\[\mathcal{X}_1 = \iota_\ast(\mathcal{X}).\]
\end{lem}
\begin{proof} This is immediate, using the fact that $\iota_\ast(P) = P\cap B_1$ for $P \in \Spec_{\gr}(B)$, and that $\iota_\ast^{-1}(\mathfrak{p}) =  \sqrt{\mathfrak{p}B}$ for $\mathfrak{p} \in \Spec_{\gr}(B_1)$.
\end{proof}

\begin{prop} Let $I$ be a proper, non-zero $\Gamma$-invariant right ideal of $kG$ such that $kG/I$ is pure. Then there exists a minimal prime $P$ above $\gr I$ and $v \in V \backslash 0$ such that $\mathfrak{g}\cdot v \subseteq P$.
\end{prop}
\begin{proof} Suppose for a contradiction that $\min \Ch(A/I) \cap \mathcal{X} = \emptyset;$ we will show that $I = 0$. By Theorem \ref{ContKG}, $I$ is controlled by $A_1$: $I = I_1A$. Since $\iota_\ast(\Ch(A/I)) = \Ch(A_1/I_1)$ by Proposition \ref{ChFrob} and since $\iota_\ast(\mathcal{X}) = \mathcal{X}_1$ by the lemma, we see that
\[\min \Ch(A_1/I_1) \cap \mathcal{X}_1 = \emptyset;\]
moreover $A_1/I_1$ is pure by Lemma \ref{Pure}(c). Thus $I_1$ satisfies the same conditions as $I$. We can now apply the above argument repeatedly and deduce that $I$ is controlled by $kG^{p^n}$ for all $n\geq 0$. Since $I$ is proper and since $kG^{p^n}$ is scalar local, $I = (I \cap kG^{p^n})kG \subseteq (G^{p^n} - 1)kG$ for all $n \geq 0$. The intersection of these augmentation ideals is zero, so $I = 0$. This is the required contradiction.
\end{proof}

\begin{cor} Let $I$ be a proper, non-zero $\Gamma$-invariant right ideal of $kG$ such that $kG/I$ is pure. Then \[j(kG/I) \geq u:= \min \{\dim \mathfrak{g}\cdot v : v \in V \backslash 0\}.\]
\end{cor}
\begin{proof} By the theorem, we can find $P \in \min \Ch(A/I)$ such that $P \in \mathcal{X}$. Then $P$ contains $(\mathfrak{g}\cdot v)B$, which is clearly a prime ideal of height $\dim \mathfrak{g}\cdot v$. Hence $\hgt P \geq \dim \mathfrak{g}\cdot v \geq u$. Since every prime in $\min \Ch(A/I)$ has the same height by Gabber's purity theorem, we can now use Lemma \ref{Pure}(b) to deduce that $j(A/I) = \hgt P \geq u$ as required.
\end{proof}

In fact, the assumption that $kG/I$ is pure is unnecessary:

\subsection{Theorem}\label{Main} Let $I$ be a proper, non-zero $\Gamma$-invariant right ideal of $kG$. Then there exists a minimal prime ideal $P$ above $\gr I$ and $v \in V \backslash 0$ such that $\mathfrak{g}\cdot v \subseteq P$. It follows that $j(kG/I) \geq \min \{\dim \mathfrak{g}\cdot v : v \in V \backslash 0\}$.
\begin{proof}
Let $\overline{I}/I$ be the largest submodule of $A/I$ of grade strictly bigger than $j(A/I)$. We claim that the right ideal $\overline{I}$ is still $\Gamma$-invariant. For any $\gamma \in \Gamma$, we saw in the proof of Proposition \ref{SourceDers} that $\gr \gamma(x) = \gr x$ for any $x \in A$. This means that $\gr \gamma(\overline{I}) = \gr \overline{I}$, so
\[j_A(\gamma(\overline{I}) / I) = j_B(\gr \gamma(\overline{I}) / \gr I) = j_B(\gr \overline{I}/\gr I) = j_A(\overline{I} / I) > j_A(A/I)\]
by Lemma \ref{Pure}(b). Since $\gamma(\overline{I})$ contains $\gamma(I) = I$, it follows that $\gamma(\overline{I}) \subseteq \overline{I}$. Hence $\overline{I}$ is $\Gamma$-invariant as claimed, it is still proper and non-zero, but now $A/\overline{I}$ is also pure.

Let $P \in \min\Ch(A/\overline{I})$; then $P \in \Ch(A/I)$ by Corollary \ref{Ch}. If $Q \subseteq P$ also lies in $\Ch(A/I)$, then $\hgt Q \geq j(A/I) = j(A/\overline{I}) = \hgt P$ by Lemma \ref{Pure}(b), forcing $Q = P$. Hence
\[\min\Ch(A/\overline{I}) \subseteq \min\Ch(A/I).\]
Therefore $\min\Ch(A/I) \cap \mathcal{X} \neq \emptyset$ by Proposition \ref{Restrictions}, and $j(A/I) = j(A/\overline{I}) \geq u$ by Corollary \ref{Restrictions}.
\end{proof}

\section{Applications}
\label{Appl}
\subsection{Linear actions}
One obvious example of a situation where our results are applicable is the case of a uniform pro-$p$ group $\Gamma$ acting linearly on some free abelian pro-$p$ group $G$ of finite rank $n$. Because $G$ is abelian, $\Gamma$ automatically acts by group automorphisms, and we only have to assume that the image of the action is contained in the first congruence subgroup of $\Aut(G) \cong \GL_n(\Zp)$ to ensure the action is uniform. 

\begin{thm} Suppose that the image of the action of $\Gamma$ is open in $\GL_n(\Zp)$. Then the only $\Gamma$-invariant prime ideals of $kG$ are the zero ideal and the augmentation ideal.
\end{thm}
\begin{proof} Without loss of generality, we may assume that the image of $\Gamma$ is equal to a congruence subgroup of $\GL_n(\Zp)$. Then it is easy to see that the normaliser Lie algebra $N_\Gamma$ defined in $\S\ref{Normaliser}$ is the full linear Lie algebra $\mathfrak{gl}_n(\Zp)$. Hence $\mathfrak{g} = N_\Gamma / pN_\Gamma = \mathfrak{gl}_n(\Fp)$ and $V = G / G^p = \Fp^n$ is the natural $\mathfrak{g}$-module. It is now clear that $\min\{\dim \mathfrak{g} \cdot v : v \in V \backslash 0\} = n$, so Theorem \ref{Main} implies that a non-zero $\Gamma$-invariant prime ideal $I$ of $kG$ must satisfy $j(kG/I) = n$. This forces $I$ to be the augmentation ideal of $kG$ since $kG$ is local, and the result follows.
\end{proof}

An obvious modification of the above proof shows that the result is also true if we only assume that the image of $\Gamma$ is an open subgroup of $\SL_n(\Zp)$. These results provide some very weak evidence towards the following conjecture, which is inspired by Roseblade's theorem \cite[Theorem D]{R}:

\begin{conj} Let $I$ be a $\Gamma$-invariant prime ideal of $kG$ which is \emph{faithful}, that is $(1 + I) \cap G = 1$. Then $I$ is controlled by the subgroup of $\Gamma$-fixed points of $G,$ $G^\Gamma$:
\[I = (I \cap kG^\Gamma)kG.\]
\end{conj}

A more conceptual proof of Roseblade's theorem can be found in \cite[\S1.3]{BG}.

\subsection{Calculations with Chevalley Lie algebras}
\label{ChevComp} Let $\Phi$ be an indecomposable root system. We fix a set of fundamental roots $\Pi$ and the corresponding set of positive roots $\Phi^+$ of $\Phi$. Let $\rho = \frac{1}{2}\sum_{s \in \Phi^+}s$ and let $\theta$ be the highest root of $\Phi^+$. Recall that the root system $\Phi$ is by definition a finite subset of some real Euclidean space, so the inner product $(2\rho, \theta)$ makes sense.

Let $\mathfrak{g}_k = \Phi(k)$ be the Chevalley Lie algebra over our field $k$ of characteristic $p$. Recall \cite[\S 4]{C} that by definition, $\mathfrak{g}_k$ has a basis consisting of the root vectors $e_s$ ($s \in \Phi$) , and the fundamental co-roots $h_r$, ($r\in\Pi$). Let $\mathfrak{b}_k$ denote the (positive) Borel subalgebra of $\mathfrak{g}_k$, spanned by all of the co-roots and the $e_s$ where $s \in \Phi^+$. 

Recall \cite[\S 0.3]{AWZ2} that we call $p$ a \emph{nice prime} for $\Phi$ if $p\geq 5$ and if $p \nmid n+1$ when $\Phi$ is the root system $A_n$. We would like to thank Alexander Premet for providing the proof of the following result.

\begin{prop} Suppose that $p$ is a nice prime for $\Phi$. Let 
\[u_k:=\min\{\dim_k [\mathfrak{g}_k,x] : x \in \mathfrak{g}_k \backslash 0\}.\]
Then $u_k = (2\rho, \theta)$.
\end{prop}
\begin{proof} We first assume that the field $k$ is algebraically closed. Let $\mathcal{G}$ denote the Chevalley group over $k$ associated to $\Phi$. This is a connected, adjoint, simple algebraic group over $k$, whose Lie algebra equals $\mathfrak{g}_k$. Hence $\mathcal{G}$ acts naturally on $\mathfrak{g}_k$ by Lie algebra automorphisms. For each $n \geq 0$, let $X(n) = \{x \in \mathfrak{g}_k : \dim [\mathfrak{g}_k,x] \leq n\}$. Then clearly
\[ u_k = \min\{n : X(n) \neq \{0\}\}.\]
Now $X(u_k)$ is a closed subset of $\mathfrak{g}_k$ in the Zariski topology, which is moreover conical, $\mathcal{G}$-stable and contains at least one line. Let $\mathcal{B}$ be the Borel subgroup of $\mathcal{G}$ whose Lie algebra equals $\mathfrak{b}_k$. By Borel's fixed point theorem \cite[Chapter III, Theorem 10.4]{Borel}, $X(u_k)$ contains a $\mathcal{B}$-stable line $\langle x\rangle$ for some nonzero $x \in \mathfrak{g}_k$. By our choice of $\mathcal{B}$, $x$ is then a highest weight vector for the $\mathcal{G}$-module $\mathfrak{g}_k$, so $x$ is also a highest weight vector for the $\mathfrak{g}_k$-module $\mathfrak{g}_k$.

Our assumption on $p$ implies that $\mathfrak{g}_k$ is a simple Lie algebra --- see \cite[p. 187]{S}. Hence this $\mathfrak{g}_k$-module is irreducible. It now follows (essentially from the PBW theorem) that $x$ must be a non-zero scalar multiple of $e_\theta$, where $\theta$ is the highest root of $\Phi^+$. Following \cite{Wang}, we call $\mathbb{S} := \{ \beta \in \Phi^+: \theta - \beta \in \Phi\}$ the set of \emph{special roots}. Since $p \geq 5$ by assumption, Chevalley's basis theorem \cite[Theorem 4.2.1]{C} implies that $[e_r, e_s] \neq 0$ whenever $r,s$ and $r+s\in\Phi$. We can therefore compute
\[ [\mathfrak{g}_k, e_\theta] = \langle e_\theta, h_\theta\rangle \oplus \langle e_{\theta - \beta} : \beta\in\mathbb{S}\rangle,\]
whence $u_k = \dim_k [\mathfrak{g}_k, e_\theta] = |\mathbb{S}| + 2$. It is shown in the proof of \cite[Lemma 3]{Wang} that $|\mathbb{S}|+2 = (2\rho, \theta)$, so $u_k = (2\rho, \theta)$ when $k$ is algebraically closed.

Returning to the general case, let $\overline{k}$ denote an algebraic closure of $k$. The computation of $[\mathfrak{g}_k, e_\theta]$ performed above does not require $k = \overline{k}$, so 
\[ (2\rho, \theta) = \dim_k [\mathfrak{g}_k, e_\theta] \geq u_k \geq u_{\overline{k}} = (2\rho, \theta)\]
and the proposition follows.
\end{proof}

\subsection{Iwasawa algebras of Chevalley type}
\label{IwaChev}
Here is our main result.

\begin{thm} Let $p$ be a nice prime for $\Phi$, let $G = \exp(p^t \Phi(\Zp))$ be the uniform pro-$p$ group of Chevalley type for some $t \geq 1$, and let $I$ be a non-zero two-sided ideal of $kG$. Then 
\[ j(kG/I) \geq (2\rho, \theta).\]
\end{thm}
\begin{proof} $\Gamma = G$ acts on $G$ by conjugation, and this action is uniform in the sense of $\S\ref{UnifAct}$ because $G$ is uniform. Clearly $I$ is a $\Gamma$-invariant right ideal of $kG$. Following the proof of \cite[Theorem 3.4]{AWZ2}, we see that $N_\Gamma = \Phi(\Zp)$ and that $\mathfrak{g} = \Phi(\Fp)$. Moreover, $x \mapsto p^tx$ induces a $\mathfrak{g}$-module isomorphism between the adjoint $\mathfrak{g}$-module $\mathfrak{g}$ and the $\mathfrak{g}$-module $V = G/G^p$.

We can now use Theorem \ref{Main} to deduce that $j(kG/I)$ is bounded below by $\min\{\dim [\mathfrak{g},x] : x \in \mathfrak{g} \backslash 0\}$. But this number equals $(2\rho, \theta)$ by Proposition \ref{ChevComp}.
\end{proof}
It can be shown \cite[Exercise 6.2]{K} that $(2\rho,\theta) = 2\hat{h} - 2$, where $\hat{h}$ is the \emph{dual Coxeter number} of $\Phi$ that arises in the study of affine Lie algebras. A list of values of this invariant can be found in \cite[\S 6.1]{K}; we used this list to construct the table given in $\S\ref{Table}$.

\end{document}